\documentclass[a4paper, reqno]{amsart}
\usepackage{amsmath, amssymb, eucal, amscd, amstext, enumerate, mathrsfs, yfonts, amsfonts, verbatim, tikz, comment}

\usepackage[plainpages=false,colorlinks,hyperindex,pdfpagemode=None,bookmarksopen,linkcolor=red,citecolor=blue,urlcolor=blue]{hyperref}

\newtheorem{thm}{Theorem}
\newtheorem{lem}[thm]{Lemma}

\newtheorem{prop}[thm]{Proposition}

\newtheorem{rem}[thm]{Remark}
\newtheorem{defn}[thm]{Definition}

\newcommand{\smnoind}{\smallskip\noindent}

\newcommand{\id}{\mathrm{id}}

\newcommand{\CE}{\mathcal{E}}

\newcommand{\KG}{\mathbb{G}}

\newcommand{\KI}{\mathfrak{I}}
\newcommand{\KJ}{\mathfrak{J}}

\newcommand{\ME}{\mathrm{E}}

\newcommand{\BT}{\mathbf{T}}

\newcommand{\I}{\mathrm{I}}

\begin{document}

\title{A note on Boolean inverse monoids and ample groupoids}

\author[Ng and Tian]{Chi-Keung Ng \and Rui Tian}

\address[Chi-Keung Ng]{Chern Institute of Mathematics and LPMC, Nankai University, Tianjin 300071, China.}
\email{ckng@nankai.edu.cn; ckngmath@hotmail.com}

\address[Rui Tian]{Chern Institute of Mathematics, Nankai University, Tianjin 300071, China.}
\email{1120200009@mail.nankai.edu.cn}



\begin{abstract}
It is a study note detailing a connection between Boolean inverse $\wedge$-monoids and ample groupoids. 
\end{abstract}

\maketitle


In this note, we present a self-contained treatment for a connection between Boolean inverse $\wedge$-monoids and ample groupoids   (assuming only \cite[Lemma 2.2(2)]{L10}), which was stated in \cite[Theorem 3.3(2)]{Steinb} in the more general situation of Boolean inverse monoids, whose proof requires to dig through \cite{Exel-inv-sg}, \cite{Exel-Reconst} and \cite{LL}. 

After we finished this note, we learn that a self-contained proof for \cite[Theorem 3.3]{Steinb} was also given in \cite[Theorem A.10]{Rena}. 
In fact, \cite[Theorem A.10]{Rena} establishes an equivalence between the category of Boolean inverse semigroups with morphisms being Boolean inverse semigroup morphisms, and the category of ample groupoids with morphisms being multiplier actions with proper anchor (which is different from \cite[Theorem 3.4]{Steinb}).

\smallskip

Let us recall that a \'{e}tale groupoid $\KG$ is \emph{ample} if its unit space $\KG^{(0)}$ is a totally disconnected compact Hausdorff space.
It is well-known that the set $\Gamma_c(\KG)$ of all compact open bisections of an ample groupoid $\KG$ forms a Boolean inverse monoid (see e.g., \cite[Proposition 2.18(8)]{L10}, \cite[Theorem 4.4(2)]{L20} or \cite{Rena}). 

\smallskip

Recall that a semigroup $S$ is an \emph{inverse semigroup} if for each $\phi\in S$, there is a unique element $\phi^{-1}\in S$ satisfying 
\begin{equation*}
	\phi\phi^{-1}\phi = \phi\qquad \text{and} \qquad \phi^{-1}\phi\phi^{-1} = \phi^{-1}.
\end{equation*}
An inverse semigroup $S$ is called an \emph{inverse monoid} if it has an identity $1$.
There is a \emph{canonical ordering} on an inverse semigroup given by 
\begin{equation*}
	\phi\leq \psi \quad \text{whenever}\quad \phi = \psi\phi^{-1}\phi.
\end{equation*}
The commutative subsemigroup 
	$$\ME(S):= \{p\in S: p^2 = p\}$$ 
of $S$ is a  meet semilattice under the canonical order; actually, 
$$p \wedge q = pq \qquad (p,q\in \ME(S)).$$

Let $S$ be an inverse semigroup with a zero, denoted by $0$.
Two elements $\phi$ and $\psi$ in $S$ are said to be \emph{orthogonal}, and denoted by $\phi\bot \psi$,  if 
$$\phi\psi^{-1} = \phi^{-1}\psi = 0.$$

Basic information on inverse semigroups can be found in \cite{Laws}.

\smallskip

\begin{defn}\label{defn:Boolean}
Let $S$ be an inverse semigroup with a zero. 

\smnoind
(a) $S$ is said to be 
	\emph{additive} (or \emph{finitely orthogonally complete}) if for $\phi,\psi\in S$ with $\phi\bot \psi$, the supremum $\phi\vee \psi$ of $\phi$ and $\psi$ in $S$ exists, and one has $\chi(\phi\vee \psi) = (\chi\phi) \vee (\chi\psi)$ as well as $(\phi\vee \psi)\chi = (\phi \chi) \vee (\psi \chi)$ for every $\chi\in S$. 

\smnoind
(b)	$S$ is called a \emph{Boolean inverse monoid} (respectively, \emph{Boolean inverse $\wedge$-monoid}) if $\ME(S)$ 
		is a Boolean algebra under the canonical ordering, and $S$ is an additive inverse semigroup 
		(respectively, which is a meet semilattice under the canonical ordering).
\end{defn}


\smallskip

Note that a semigroup $S$ will have an identity when $\ME(S)$ is a Boolean algebra. 

\smallskip

By the discussion concerning finite orthogonal completeness after Condition (BM3) in \cite[p.386]{L10}, it suffices to assume, in the definition of Boolean inverse $\wedge$-monoid, the existence of $\phi \vee \psi$ whenever $\phi\bot \psi$, instead of additivity as above. 
Thus, our notion of Boolean inverse $\wedge$-monoids coincides with the one in \cite{LS}, and coincide with notion of Boolean inverse monoids in \cite{L10, L12,LL}.

\smallskip

Let $S$ be a Boolean inverse monoid.
The construction as in \cite{Rena} and \cite{Steinb} associates with $S$ an ample groupoid (a similar construction can be found in  \cite[Theorem 3.3.2]{Pate} and \cite[\S 2.6]{STY}). 
Let us recall this construction in the following. 

It is well-known that the topological space $\widehat{\ME(S)}$ of characters on the Boolean algebra $\ME(S)$ is a totally disconnected compact Hausdorff space  (see e.g., Lemma 1 of  \cite[Chapter 34]{GH}).  
We set 
$$\widehat{\ME(S)}_\phi:= \big\{x\in \widehat{\ME(S)}: x(\phi^{-1}\phi) =1 \big\}\qquad (\phi\in S).$$
For each $\phi\in S$, there is a homeomorphism $\alpha_\phi:\widehat{\ME(S)}_\phi\to \widehat{\ME(S)}_{\phi^{-1}}$
given by 
\begin{equation}\label{eqt:homeo-alpha-phi}
\alpha_\phi(x)(p) := x(\phi^{-1}p\phi) \qquad \big(x\in \widehat{\ME(S)}_\phi; p\in \ME(S)\big).
\end{equation}

\smallskip

Let us put 
$$S\ast \widehat{\ME(S)}:= \big\{(\phi,x)\in S\times \widehat{\ME(S)}: x\in \widehat{\ME(S)}_\phi \big\}.$$
Define an equivalence relation $\sim$ on $S\ast \widehat{\ME(S)}$ such that $(\phi,x)\sim (\psi,y)$ if and only if
$$x=y \text{ and there is $p\in \ME(S)$ with $x(p) = 1$ and $\phi p = \psi p$}.$$
Denote by $[\phi,x]$ the equivalence class under $\sim$ that contains $(\phi,x)$. 
The set 
$$\KG(S):= S\ast \widehat{\ME(S)}/\sim$$ 
of all such equivalence classes is a groupoid such that (see \eqref{eqt:homeo-alpha-phi})
$$\big([\phi,x], [\psi,y]\big)\in \KG(S)^{(2)}\quad \text{if and only if} \quad x = \alpha_\psi(y),$$ 
with the multiplication given by 
\begin{equation}\label{eqt:prod-coarse-gpd}
[\phi,x]\cdot [\psi,y] := [\phi\psi,y] \qquad \big(\big([\phi,x], [\psi,y]\big)\in \KG(S)^{(2)}\big).
\end{equation}
In this case, one has 
\begin{equation}\label{eqt:base-char}
	\KG(S)^{(0)} = \big\{[1,x]\in \KG(S): x\in \widehat{\ME(S)} \big\}
\end{equation}
as well as
\begin{equation}\label{eqt:inver}
[\phi,x]^{-1} =  [\phi^{-1}, \alpha_\phi(x)].
\end{equation}

\smallskip

We will consider the topology on $\KG(S)$ generated by 
the collection $\{ U_\phi: \phi\in S\}$ of subsets, 
where 
$$U_\phi:= \big\{[\phi,x]\in \KG(S): x\in \widehat{\ME(S)}_\phi \big\}\qquad (\phi\in S).$$

\smnoind

In Lemma \ref{lem:from-sg-to-gpd}(g) below, we will show that $\KG(S)$ is an ample groupoid. 
The arguments leading to this fact will also be needed in the proof of Proposition \ref{prop:Bool-inv-sg-to-ample-gpd}. 

\smnoind

Let us first present a lemma that depends on the following result.

\smallskip

\begin{lem}[{\cite[Lemma 2.2(2)]{L10}}]\label{lem:compliment}
Let $S$ be a Boolean inverse $\wedge$-monoid. 
If $\phi,\chi\in S$ satisfying $\chi\leq \phi$, then  there is a unique element 
$\phi\setminus \chi\in S$ such that 
$$(\phi\setminus \chi)\bot \chi\quad \text{and} \quad \phi = (\phi\setminus \chi)\vee \chi.$$ 
\end{lem}

\smallskip

\begin{lem}\label{lem:intersection}
Let $S$ be a Boolean inverse $\wedge$-monoid, $\phi,\psi\in S$ and $x\in \widehat{\ME(S)}$. 

\smnoind
(a) Suppose $\psi\leq \phi$ (i.e., $\psi = \phi \psi^{-1}\psi$). 
If $[\psi, x]\in U_\psi$, then $(\phi,x)\in S\ast \widehat{\ME(S)}$ and 
$[\phi,x] = [\psi,x]$. 
Moreover, $U_\psi \subseteq U_\phi$. 

\smnoind
(b) If $\ME(S)_{\psi,\phi}:= \{p\in \ME(S): \phi p = \psi p \},$
then  $U_\phi\cap U_\psi = {\bigcup}_{q\in \ME(S)_{\psi,\phi}} U_{\phi q} = U_{\phi\wedge \psi}$. 

\smnoind
(c) If $\phi\bot \psi$, then $U_\phi\cap U_\psi = \emptyset$ and $U_\phi \cup U_\psi = U_{\phi\vee \psi}$. 

\smnoind
(d) $U_\phi \setminus U_\psi = U_{\phi \setminus (\phi\wedge \psi)}$, where $\phi \setminus (\phi\wedge \psi)$ is as in Lemma \ref{lem:compliment}.
\end{lem}
\begin{proof}
(a) It follows from $\psi^{-1}\psi \leq \phi^{-1}\phi$ that $\widehat{\ME(S)}_\psi\subseteq \widehat{\ME(S)}_\phi$, which implies $(\phi,x)\in S\ast \widehat{\ME(S)}$. 
Moreover, since 
$$\psi \psi^{-1}\psi  = \psi = \phi \psi^{-1}\psi,$$ 
$\psi^{-1}\psi\in \ME(S)$ and $x(\psi^{-1}\psi) = 1$, one has $(\phi,x)\sim (\psi,x)$.
Hence, 
\begin{equation*}
[\psi,x] = [\phi,x]\in U_\phi \quad \text{for every} \quad [\psi, x]\in U_\psi.
\end{equation*}

\smnoind
(b) Pick any $z\in \widehat{\ME(S)}_\phi$ and $y\in \widehat{\ME(S)}_\psi$ that satisfy $[\phi,z] = [\psi,y]\in U_\phi\cap U_\psi$. 
Then $z=y$ and there exists $q\in \ME(S)$ with $z(q) = 1$ and $\phi q = \psi q$. 
This means that $q\in \ME(S)_{\psi,\phi}$ and $[\phi,z] = [\phi q, z]\in U_{\phi q}$; note that $(\phi q,z)\in S\ast \widehat{\ME(S)}$ as 
$$q\phi^{-1}\phi q = \phi^{-1}\phi\wedge q\quad \text{and} \quad z(\phi^{-1}\phi\wedge q) = z(\phi^{-1}\phi)\wedge z(q) =1.$$ 
Secondly, if $p\in \ME(S)_{\psi,\phi}$, then $\phi p = \psi p \leq \phi\wedge \psi$, and part (a) above implies that $U_{\phi p}\subseteq U_{\phi\wedge \psi}$. 
Finally, it follows from part (a) above that $U_{\phi\wedge \psi}\subseteq U_\phi\cap U_\psi$. 

\smnoind
(c) Note that $\phi\bot \psi$ implies $\phi\wedge \psi = 0$, because $\chi\leq \phi\wedge \psi$ gives 
$$\chi^{-1}\chi\leq \phi^{-1}\phi\wedge \psi^{-1}\psi = \phi^{-1}\phi\psi^{-1}\psi = 0.$$ 
Thus, it follows from part (b) above that $U_\phi\cap U_\psi = \emptyset$. 

On the other hand, as $\phi, \psi \leq \phi\vee \psi$, we know from part (a) above that $U_\phi \cup U_\psi \subseteq U_{\phi\vee \psi}$. 
Conversely, let $[\phi\vee \psi, z]\in U_{\phi\vee \psi}$. 
Since $\phi^{-1}\psi = 0 = \psi^{-1}\phi$ and $S$ is additive, 
$$1 = z\big((\phi\vee \psi)^{-1}(\phi\vee \psi)\big) = z\big((\phi^{-1}\vee \psi^{-1})(\phi\vee \psi)\big) = z(\phi^{-1}\phi) \vee z(\psi^{-1}\psi).$$ 
If $z(\phi^{-1}\phi) = 1$, then $[\phi,z]\in U_\phi$, and part (a) above produces 
$[\phi\vee \psi, z] = [\phi, z]\in U_\phi$. 
Similarly, if $z(\psi^{-1}\psi) = 1$, then $[\phi\vee \psi, z] \in U_\psi$. 

\smnoind
(d) As $\big(\phi \setminus (\phi\wedge \psi)\big)\bot (\phi\wedge \psi)$ and $\phi = \big(\phi \setminus (\phi\wedge \psi)\big)\vee (\phi\wedge \psi)$, part (c) implies 
$$U_{\phi \setminus (\phi\wedge \psi)}\cap U_{\phi\wedge \psi} = \emptyset \quad \text{and} \quad U_\phi = U_{\phi \setminus (\phi\wedge \psi)}\cup U_{\phi\wedge \psi}.$$
This, together with part (b) above, gives 
$$U_{\phi \setminus (\phi\wedge \psi)} = U_\phi \setminus U_{\phi \wedge \psi} = U_\phi \setminus \big(U_\phi \cap  U_\psi\big) = U_\phi \setminus U_\psi,$$
as asserted.
\end{proof}

\smallskip

\begin{lem}\label{lem:from-sg-to-gpd}
Let $S$ be a Boolean inverse $\wedge$-monoid and $\phi\in S$.

\smnoind
(a) $\{ U_\psi: \psi\in S\}$ is a basis for the topology on $\KG(S)$. 

\smnoind
(b) If $p \in \ME(S)$, then $U_{\phi p}= \big\{[\phi, x]: x\in \widehat{\ME(S)}_{\phi^{-1}\phi p}\big\}$.

\smnoind
(c) If $V\subseteq \widehat{\ME(S)}_{\phi^{-1}\phi}$ is an open subset, then $\big\{[\phi,x]\in \KG(S): x\in V \big\}$
is an open subset of $U_\phi$.

\smnoind
(d) If $W$ is an open subset of $U_\phi$, then there exists an open subset $\tilde W \subseteq \widehat{\ME(S)}_{\phi^{-1}\phi}$ such that  $W = \big\{[\phi,x]\in \KG(S): x\in \tilde W \big\}$. 

\smnoind
(e) The map $\theta_\phi:U_\phi\to \widehat{\ME(S)}_{\phi^{-1}\phi} = \widehat{\ME(S)}_{\phi}$ 
sending  $[\phi,x]$ to $x$ is a homeomorphism, and 
$\theta_\phi$ can be identified with the restriction of the source map on $U_\phi$ 
when $\widehat{\ME(S)}$ is identified with 
$\KG(S)^{(0)}$, via Equality \eqref{eqt:base-char}. 

\smnoind
(f) $U_\phi\in \Gamma_c(\KG(S))$.

\smnoind
(g) $\KG(S)$ is an ample groupoid. 
\end{lem}
\begin{proof}
(a) This follows directly from Lemma \ref{lem:intersection}(b).

\smnoind
(b) It follows from Lemma \ref{lem:intersection}(a) that 
\begin{align*}\label{eqt:U-phi-p}
	U_{\phi p}
	& = \big\{[\phi p, x]: x\in \widehat{\ME(S)}_{\phi p}\big\}\nonumber\\
	&  = \big\{[\phi, x]: x\in \widehat{\ME(S)}_{\phi p}\big\}
	= \big\{[\phi, x]: x\in \widehat{\ME(S)}_{\phi^{-1}\phi p}\big\}. 
\end{align*}
(the third equality follows from $p\phi^{-1}\phi p = \phi^{-1}\phi p$).

\smnoind
(c) As $\big\{\widehat{\ME(S)}_e:e \in \ME(S)\big\}$
is a basis for the topology on $\widehat{\ME(S)}$ (see e.g., \cite[p.328]{GH}), there is a family $\{e_j\}_{j\in \KJ}$ in $\ME(S)$ with
$$V = {\bigcup}_{j\in \KJ} \widehat{\ME(S)}_{e_j}.$$
If $j\in \KJ$, then the condition $\widehat{\ME(S)}_{e_j}\subseteq V \subseteq \widehat{\ME(S)}_{\phi^{-1}\phi}$ implies that $x(e_j) = x(\phi^{-1}\phi e_j)$ for each  $x \in \widehat{\ME(S)}$; which is equivalent to $e_j = \phi^{-1}\phi e_j$.
Thus, part (b) above gives 
$$\big\{[\phi,x]\in \KG(S): x\in V \big\} = {\bigcup}_{j\in \KJ} \big\{[\phi, x]: x\in \widehat{\ME(S)}_{\phi^{-1}\phi e_j}\big\} = {\bigcup}_{j\in \KJ} U_{\phi e_j},$$
and hence 
$$\big\{[\phi,x]\in \KG(S): x\in V \big\}$$ 
is an open subset of $U_\phi$.

\smnoind
(d) Part (a) above produces a family $\{\psi_i\}_{i\in \KI}$ in $S$ with 
$$W = {\bigcup}_{i\in \KI} U_{\psi_i}.$$
Set $\tilde W:= {\bigcup}_{i\in \KI} {\bigcup}_{p\in \ME(S)_{\psi_i,\phi}} \widehat{\ME(S)}_{\phi^{-1}\phi p}$.
Clearly,  $\tilde W$ is an open subset of $\widehat{\ME(S)}_{\phi^{-1}\phi}$. 
For each $i\in \KI$, it follows from  $U_{\psi_i}\subseteq W \subseteq U_\phi$ and Lemma \ref{lem:intersection}(b) that 
$$U_{\psi_i} =  U_\phi \cap U_{\psi_i}= {\bigcup}_{p\in \ME(S)_{\psi_i,\phi}} U_{\phi p}.$$
This and part (b) above imply that 
\begin{align*}
	W 
	& = {\bigcup}_{i\in \KI} {\bigcup}_{p\in \ME(S)_{\psi_i,\phi}} U_{\phi p}\\ 
	& = \Big\{ [\phi,x]: x\in {\bigcup}_{i\in \KI} {\bigcup}_{p\in \ME(S)_{\psi_i,\phi}} \widehat{\ME(S)}_{\phi^{-1}\phi p}\Big\}
	= \big\{[\phi,x]: x\in \tilde W \big\}.
\end{align*}

\smnoind
(e) The first claim follows from parts (c) and (d) above, while the second claim follows from that $[\phi,x]^{-1} \cdot [\phi,x] = [\phi^{-1}, \alpha_\phi(x)] \cdot [\phi,x] = [1,x]$. 

\smnoind
(f) Since $\widehat{\ME(S)}_{\phi^{-1}\phi}$ is a closed subset of the compact Hausdorff space $\widehat{\ME(S)}$, part (e) above ensures that $U_\phi$ is compact.
On the other hand, since the range of $[\phi,x]$ is the source of $[\phi,x]^{-1}$, we know from \eqref{eqt:inver} that 
\begin{equation}\label{eqt:U-phi-inv}
	U_{\phi^{-1}} = \big\{[\phi^{-1},y]\in \KG(S): y\in \widehat{\ME(S)}_{\phi^{-1}} \big\} = \big\{[\phi,x]^{-1}\in \KG(S): x\in \widehat{\ME(S)}_\phi \big\}.
\end{equation}
Therefore, by parts (a) and (e) above,  $U_\phi$ is a compact open bisection of $\KG(S)$.  

\smnoind
(g) By part (a) above, together with Equalities \eqref{eqt:prod-coarse-gpd} and \eqref{eqt:inver}, we know that $\KG(S)$ is a topological groupoid. 
Moreover, parts (e) and (f) above tell us that the topological groupoid $\KG(S)$ is an \'{e}tale groupoid, and hence is an ample groupoid, because the compact Hausdorff space $\widehat{\ME(S)}$ is totally disconnected. 
\end{proof}

\smallskip

Our next remark presents a connection between the construction of $\KG(S)$ and the construction of the coarse groupoid of a uniformly locally finite coarse space. 
For details about uniformly locally finite coarse spaces and coarse groupoids, the reader can refer to \cite{Roe,STY}.

\smallskip

\begin{rem}\label{rem:coarse-groupoid}
For a set $X$, it is well-known that the set $\I(X)$ of all partial bijections on $X$ (recall that a partial bijection is a bijection from a subset of $X$ onto a subset of $X$) is a Boolean inverse $\wedge$-monoid. 

Suppose that $(X,\CE)$ is a uniformly locally finite coarse space. 
Let $\BT_\CE$ be the set of all ``partial translations'' as in \cite[p.63]{Roe} (i.e., the set of elements in $\I(X)$ with their graphs being elements in $\CE$). 
It can be shown that  $\BT_\CE$ is a Boolean inverse $\wedge$-monoid (see, e.g., \cite[Lemma 4.5(a)]{NT-coarse-inv}). 
Moreover, we have  
\begin{equation*}
	\ME(\BT_\CE) = \ME(\I(X)) = \{\id_Y:Y\subseteq X\}.
\end{equation*}
This implies that $\widehat{\ME(\BT_\CE)}$ can be identified with the Stone-\v Cech compactification $\beta X$ of $X$. 
By Lemma \ref{lem:from-sg-to-gpd}(e), the ``source map'' 
$$\theta_\phi:U_\phi\to \widehat{\ME(\BT_\CE)}_{\phi^{-1}\phi}$$ 
is a homeomorphism for every $\phi\in \BT_\CE$. 

From this, we see that $\KG(\BT_\CE)$ coincides with the topological groupoid $\mathnormal{G}(\mathscr{G}(X))$ as constructed in \cite[\S 2.6]{STY} with respect to  the inverse semigroup  $\mathscr{G}(X)$ as  in \cite[Definition 3.1]{STY} (notice that the inverse semigroup $\mathscr{G}(X)$ is the same as the inverse semigroup $\varGamma_\CE$  defined in the paragraph preceding \cite[Lemma 2.8]{STY}, which is exactly our inverse semigroup $\BT_\CE$). 
Therefore, $\KG(\BT_\CE)$ coincides with the coarse groupoid $\KG_\CE$ for $(X,\CE)$ (because of  \cite[Proposition 3.2]{STY}).
\end{rem}

\smallskip

The following result was stated in \cite[Theorem 3.3(2)]{Steinb} (in the more general situation of Boolean inverse monoids), and it claimed \cite{Exel-inv-sg}, \cite{Exel-Reconst} and \cite{LL} as its references. 
We present here a self-contained proof (see  also \cite[Theorem A.10]{Rena}).

\smallskip

\begin{prop}\label{prop:Bool-inv-sg-to-ample-gpd}
Let $S$ be a Boolean inverse $\wedge$-monoid. 
The map $\varepsilon:S \to \Gamma_c(\KG(S))$ given by (see Lemma \ref{lem:from-sg-to-gpd}(f)) 
$$\varepsilon(\phi) := U_\phi$$
is a semigroup isomorphism. 
\end{prop}
\begin{proof}
For every $\phi, \psi \in S$, one has 
\begin{align*}
	\varepsilon(\phi) \varepsilon(\psi) & = \big\{[\phi\psi,y] \in \KG(S): y \in \widehat{\ME(S)}_{\psi}; \alpha_\psi(y) \in \widehat{\ME(S)}_{\phi}\big\} \\
	& = \big\{[\phi\psi,y] \in \KG(S): y \in \widehat{\ME(S)}; y(\psi^{-1} \psi) = y(\psi^{-1}\phi^{-1}\phi\psi) = 1\big\} \\
	& = \big\{[\phi\psi,y] \in \KG(S): y \in \widehat{\ME(S)}; y(\psi^{-1}\phi^{-1}\phi\psi) = 1\big\} \\
	& = \varepsilon(\phi\psi)
\end{align*}
(as $\psi^{-1}\phi^{-1}\phi\psi \leq \psi^{-1}\psi$). 
This shows that $\varepsilon$ is a semigroup homomorphism. 
 
For the injectivity of $\varepsilon$, assume that $\phi,\psi\in S$ with $U_\phi = U_\psi$. 
By Lemma \ref{lem:intersection}(d),  
$$U_{\phi \setminus (\phi\wedge \psi)} = \emptyset,$$
which produces $\phi \setminus (\phi\wedge \psi) = 0$. 
Hence, $\phi = \phi \wedge \psi$ (see Lemma \ref{lem:compliment}). 
By symmetry, one also has $\psi = \phi \wedge \psi$.

For the surjectivity, we consider $W\in \Gamma_c(\KG(S))$. 
Lemma \ref{lem:from-sg-to-gpd}(a) gives a family $\{\phi_i\}_{i\in \KI}$ of elements in $S$ such that
$$W = {\bigcup}_{i\in \KI}U_{\phi_i}.$$ 
Since $W$ is compact and $U_{\phi_i}$ $(i \in \KI)$ are open, we obtain a finite subset $\{\phi_1,\dots,\phi_n\} \subseteq \{\phi_i\}_{i\in \KI}$ with 
$W = {\bigcup}_{k=1}^n U_{\phi_k}.$
It suffices to find $\psi\in S$ satisfying
$${\bigcup}_{k=1}^n U_{\phi_k} = U_\psi.$$ 
Indeed, Lemma \ref{lem:intersection}(d) implies 
$$U_{\phi_1} \cup U_{\phi_2} = U_{\phi_1} \cup \big(U_{\phi_2}\setminus U_{\phi_1}\big) = U_{\phi_1} \cup U_{\chi},$$
where $\chi:= \phi_2\setminus (\phi_2 \wedge \phi_1)$. 
As $U_{\phi_1}$ and $U_{\chi} = U_{\phi_2}\setminus U_{\phi_1}$ are disjoint, and the source map is injective on $W$, Lemma \ref{lem:from-sg-to-gpd}(e) gives  $\widehat{\ME(S)}_{\phi_1^{-1}\phi_1}\cap \widehat{\ME(S)}_{\chi^{-1}\chi} = \emptyset$.  
Hence, 
$$\phi_1^{-1}\phi_1 \wedge \chi^{-1}\chi = 0.$$ 
By considering the injectivity of the range map on $W$, one can use Lemma \ref{lem:from-sg-to-gpd}(f) (and in particular, \eqref{eqt:U-phi-inv})  to show that  
$$\widehat{\ME(S)}_{\phi_1\phi_1^{-1}}\cap \widehat{\ME(S)}_{\chi\chi^{-1}} = \emptyset,$$ 
and the same argument as above gives $\phi_1\phi_1^{-1} \wedge \chi\chi^{-1} = 0$. 
Therefore, $\phi_1\bot \chi$.
It then follows from Lemma \ref{lem:intersection}(c) that 
$$U_{\phi_1} \cup U_{\phi_2} = U_{\phi_1\vee \chi}.$$
Inductively, one obtains the required element $\psi\in S$. 
\end{proof}

\smallskip

By Remark \ref{rem:coarse-groupoid} and Proposition \ref{prop:Bool-inv-sg-to-ample-gpd}, if $(X,\CE)$ is a uniformly locally finite coarse space, then one can recover the inverse monoid $\BT_\CE$ from the coarse groupoid $\KG_\CE$. 
Using this and \cite[Theorem 4.7]{NT-coarse-inv}, we can construct $(X,\CE)$ from $\KG_\CE$. 
Furthermore, as established in \cite{NT-coarse-inv}, many studies of the coarse space $(X,\CE)$ can be done through the corresponding studies of the Boolean inverse $\wedge$-monoid $\BT_\CE$. 

\smallskip

\section*{Acknowledgement}

\smallskip

We are grateful to Professor Jean Renault for sharing his article \cite{Rena}, in which Theorem A.10 establishes a stronger result than the one presented in this note.

\end{document}